\documentclass[review]{elsarticle}

\usepackage{hyperref}
\usepackage{amsmath, amssymb,amsthm}
\usepackage{mathrsfs}
\usepackage{stmaryrd}
\date{}

\def\F{\mathbb{F}}
\def\twodim{L}
\def\lbrak{\left[}
\def\rbrak{\right]}
\def\Uenv{\mathscr{U}}
\def\envalg{\Uenv}
\def\envalgDef{\envalg:=\Uenv(\twodim)}
\def\N{\mathbb{N}}
\def\twodimComp{\mathit{\Gamma}}
\def\sub{\subseteq}
\def\arbLie{\mathfrak{g}}
\def\genset{X}
\def\monoid{\left<\genset\right>}
\def\freeAlgFull{\F\monoid}
\def\freeAlg{\mathfrak{A}}
\def\freeAlgDef{\freeAlg:=\freeAlgFull}
\def\algI{1}
\def\freeLie{\mathfrak{L}}

\def\lreg{\llbracket}
\def\rreg{\rrbracket}

\def\regBasis{B}
\def\vuLetter{B}
\def\vmun{\vuLetter_1}
\def\vumvun{C}
\def\vun{D}
\def\arbregsub{S}

\def\setdiff{\backslash}
\def\Zplus{\N\setdiff\{0\}}

\def\infId{\mathcal{A}}
\def\defId{\mathcal{B}}

\def\ad{{\rm ad}\ }
\def\fchar{{\rm char}\ }
\def\span{{\rm Span}_\F\ }

\def\lpar{\left(}
\def\rpar{\right)}

\def\Q{\mathbb{Q}}
\def\spanQ{{\rm Span}_\Q\ }

\def\relcoeff{a}
\def\infgen{\alpha}

\def\ltri{\vartriangleleft}
\def\rtri{\vartriangleright}

\def\endings{E}

\def\uvLie{\mathscr{R}}

\def\arbsubF{\mathcal{K}}

\def\spanArb{{\rm Span}_\arbsubF\ }

\def\xpoly{\F\lbrak x\rbrak}
\def\dirsum{\oplus}

\def\filter{\mathcal{F}}

\newtheorem{theorem}{Theorem}[section]

\newtheorem{lemma}[theorem]{Lemma}
\newtheorem{proposition}[theorem]{Proposition}

\newtheorem{assumption}[theorem]{Assumption}

\theoremstyle{definition}
\newtheorem{definition}[theorem]{Definition}
\newtheorem{remark}[theorem]{Remark}

\makeatletter
\def\ps@pprintTitle{%
 \let\@oddhead\@empty
 \let\@evenhead\@empty
 \def\@oddfoot{}%
 \let\@evenfoot\@oddfoot}
\makeatother

\begin{document}

\begin{frontmatter}

\title{Ubiquitous Lie polynomials in a two-generator universal enveloping algebra}%\tnoteref{mytitlenote}}
%\tnotetext[mytitlenote]{Fully documented templates are available in the elsarticle package on \href{http://www.ctan.org/tex-archive/macros/latex/contrib/elsarticle}{CTAN}.}

%% Group authors per affiliation:
\author{Rafael Reno S. Cantuba}
\ead{rafael\_cantuba@dlsu.edu.ph}
\address{Mathematics and Statistics Department, De La Salle University, Manila}

\begin{abstract}
The universal enveloping algebra $\envalg$ of a two-dimensional nonabelian Lie algebra $\twodim$ is a Lie algebra itself with the commutator as Lie bracket. There exists a presentation of $\envalg$ with generators $x,y$ and relation $xy-yx=x$ such that the Lie subalgebra of $\envalg$ generated by $x,y$ is isomorphic to $\twodim$, which is only a two-dimensional vector subspace of the infinite-dimensional $\envalg$. Much then of the Lie structure of $\envalg$ is ubiquitous, yet unexamined when the characteristic of the scalar field is zero. In such a case, we show that there exists a linear complement of $\twodim$ in $\envalg$ that contains an infinite-dimensional Lie subalgebra of $\envalg$ for which we give a presentation by generators and relations. We extend this Lie subalgebra into a filtration of $\envalg$.
\end{abstract}


\end{frontmatter}

%\linenumbers

\section{Introduction}

Let $\F$ be a field, and let $\twodim$ be a non-abelian, two-dimensional Lie algebra over $\F$. A well-known result in the theory of low-dimensional Lie algebras is that $\twodim$ has a basis consisting of elements $x$ and $y$ such that
\begin{eqnarray}
\lbrak x,y\rbrak = x.\nonumber
\end{eqnarray}
See, for instance, \cite[Section 3.1]{Erd06}. Let $\envalgDef$ be the universal enveloping algebra of $\twodim$. That is, $\envalg$ is the (unital associative) algebra over $\F$ that has a presentation with generators $x,y$ and relation
\begin{eqnarray}
xy-yx=x.\label{Udefrel}
\end{eqnarray}
We turn $\envalg$ into a Lie algebra over $\F$ with Lie bracket given by $\lbrak f,g\rbrak:=fg-gf$ for any $f,g\in\envalg$. Arguably the most important property of $\envalg$, by the definition of a universal enveloping algebra, is that the Lie subalgebra of $\envalg$ generated by $x,y\in\envalg$ is isomorphic to $\twodim$ \cite[Corollary 17.3B]{hum80}, which justifies our abuse of notation, identifying $x\in\twodim$ with $x\in\envalg$, and $y\in\twodim$ with $y\in\envalg$.
 
Denote the set of all nonnegative integers by $\N$. By rewriting the relation \eqref{Udefrel} into the ``reordering formula''
\begin{eqnarray}
yx = xy -x,\label{xyReorder}
\end{eqnarray}
we can then make use of the \emph{Diamond Lemma for Ring Theory} \cite[Theorem 1.2]{ber78} to deduce that the elements
\begin{eqnarray}
x^ky^l,\quad\quad\quad(k,l\in\N),\label{Ubasis}
\end{eqnarray}
form a basis for $\envalg$. Consequently, if we define $\twodimComp\sub\envalg$ as the $\F$-linear span\textemdash the collection of all \emph{finite} linear combinations with scalar coefficients from $\F$\textemdash of all basis elements in \eqref{Ubasis} satisfying the condition $k+l\neq 1$, then we have the direct sum decomposition
\begin{eqnarray}
\envalg = \twodim\oplus\twodimComp,
\end{eqnarray}
where $\dim_\F \twodim = 2$ and $\dim_\F \twodimComp= \aleph_0=|\N|$. Contemplating on these vector space dimensions, the Lie algebra $\envalg$, primarily motivated by the low-dimensional $\twodim$, is a Lie algebra, much about the Lie algebra structure, or simply \emph{Lie structure}, of which, can still be elucidated by a closer look at $\twodimComp$. 

The Lie structure of a universal enveloping algebra has been studied in \cite{ril93}, which, as indicated in \cite[p.~47]{ril93} is based on results from \cite{pas90,pet91} which characterize restricted Lie algebras\textemdash for which the underlying field has positive characteristic\textemdash whose corresponding universal enveloping algebras satisfy some polynomial identities. The study \cite{pas90} is based on the algebraic techniques developed in \cite{ber90} which are for algebraic structures over fields with positive characteristic, while the study \cite{pet91}, similarly with \cite{ber90}, are motivated in part by \cite{bac74}, which gives some characterization for a restricted Lie algebra to be a \emph{PI-algebra}.

Emphasis is laid, in the aforementioned literature about the Lie structure of a universal enveloping algebra, on the positive characteristic of the underlying field. Even in the literature motivated by \cite{ril93}, such as \cite{alv17,sic07,sic13,sic15,use13}, which even involve other classes of associative algebras, the said type of Lie structure is investigated for the case when the field of scalars is of positive characteristic. In this study, we initiate the study for fields with characteristic zero. We also explore the role, in the study of this kind of Lie structure, of the theory of bases for free Lie algebras, in particular, Shirshov's basis \cite{shi09a,shi09b}.  

If $\arbLie$ is a Lie algebra over $\F$ and if $f_1,f_2,\ldots,f_n\in\arbLie$, we call any element of the Lie subalgebra of $\arbLie$ generated by $f_1,f_2,...,f_n$ as a \emph{Lie polynomial in $f_1,f_2,\ldots,f_n$}. In the Lie algebra $\envalg$, the Lie polynomials in $x,y$, because of the embedding of $\twodim$ into $\envalg$ are, albeit the main motivation for $\envalg$, now a triviality. Any such Lie polynomial is merely $c_1x+c_2y$ for some $c_1,c_2\in\F$. In this work, we turn our attention to the Lie structure of $\envalg$ as can be explored according to the properties of $\twodimComp$. We show that $\twodimComp$ contains an infinite-dimensional Lie subalgebra of $\envalg$ for which we give a presentation by generators and relations, and through which we describe $\envalg$ as a filtered Lie algebra. The Lie algebra $\envalg$, that was constructed from the low-dimensional $\twodim$, has a Lie structure far richer than that used to construct it. The Lie polynomials in $\envalg$ that are outside $\twodim$ form a Lie algebra, the vector space dimension of which, in comparison to that of $\twodim$, differs by an infinity. The former Lie polynomials  are then, with respect to the totality of the entire structure $\envalg$, ubiquitous.

\section{Preliminaries}

Let $X$ be a nonempty finite set. If $w=\left(x_n\right)_{n=1}^k$ is a finite sequence of elements from $\genset$, then we use the notation of writing $w$ as a juxtaposition of the sequence terms with the indices increasing, i.e., $w=x_1x_2\cdots x_k$. In such a case, we call $w$ a \emph{word on} $\genset$, or a \emph{word in} the elements of $\genset$, or simply a \emph{word}. Conversely, whenever we write a juxtaposition $x_1x_2\cdots x_k$ of sequence terms and we refer to it as a word, we always mean $x_i\in\genset$ for any $i\in\{1,2,\ldots,k\}$. We denote by $\monoid$ the set of all words on $\genset$. The set $\monoid$ is a monoid with respect to the operation of \emph{concatenation}, by which we mean the mapping $\monoid\times\monoid\rightarrow\monoid$ given by the rule
\begin{eqnarray}
\left(x_1x_2\cdots x_k,y_1y_2\cdots y_h\right)\mapsto x_1x_2\cdots x_ky_1y_2\cdots y_h.\nonumber
\end{eqnarray} 
With respect to  this monoid structure, the identity element is the empty sequence, which we denote by $\algI$. Let $\freeAlgDef$ be the free (unital associative) algebra over $\F$ generated by $\genset$. The elements of $\monoid$ form a basis for $\freeAlg$, and the multiplication operation in $\freeAlg$ is the natural extention of the concatenation operation in $\monoid$, in which the product of two arbitrary elements of $\freeAlg$ is determined by the product of the basis elements from $\monoid$. The unity element of $\freeAlg$ is the identity element of the monoid $\monoid$.

We turn $\freeAlg$ into a Lie algebra over $\F$ with Lie bracket $\lbrak f,g\rbrak:=fg-gf$ for any $f,g\in\freeAlg$. The Lie subalgebra $\freeLie$ of $\freeAlg$ generated by $\genset$ is the free Lie algebra on $\genset$ \cite[Theorem~0.5]{reu93}. Furthermore, $\freeLie$ is a proper Lie subalgebra of $\freeAlg$. i.e., Not all elements of $\freeAlg$ are Lie polynomials in (the elements of ) $\genset$. See, for instance, \cite[Theorem~1.4]{reu93} for some standard characterizations of Lie polynomials in $\genset$.

Given a nonempty word $w=x_1x_2\cdots x_m$, the \emph{length of $w$} is the positive integer $|w|:=m$, and we define the length of $\algI$ as zero. Given two words $w=x_1x_2\cdots x_{|w|}$ and $w'=y_1y_2\cdots y_{|w'|}$, because of the linear independence of the elements of $\monoid$ in $\freeAlg$, the equality $w=w'$ is characterized by the condition that $|w|=|w'|$ and $x_i=y_i$ for any $i\in\{1,2,\ldots,|w|\}$. 

Suppose $>$ is a well-ordering on $\genset$, with inverse relation denoted by $<$.  Let $w=x_1x_2\cdots x_{|w|}$ and $w'=y_1y_2\cdots y_{|w'|}$ denote two arbitrary elements of $\monoid$. We extend the relation $>$ by $w>w'$ if and only if there exists a positive integer $i$ that does not exceed $\min\{|w|,|w'|\}$ such that $x_i>y_i$ and $x_j=y_j$ for any positive integer $j<i$. Given a positive integer $i\leq |w|$, the word $x_ix_{i+1}\cdots x_{|w|}$ is called an \emph{ending of $w$}, and is a \emph{proper ending of $w$} if $i>1$. The reader can infer the analogous definitions for a \emph{beginning of $w$} and a \emph{proper beginning of $w$}. A \emph{subword} of $w$ is either a beginning or an ending of $w$. The reader may also infer from here the meaning of \emph{proper subword}. A limitation of the said extension of the relation $>$ is that there is no sensible conclusion if one of $w$ and $w'$ is a proper subword of the other. Thus, we define the relation $\rtri$ on $\monoid$ by the rule $w\rtri w'$ if and only if $ww'>w'w$. Naturally, we denote the inverse relation of $\rtri$ by $\ltri$. %The empty word $\algI$ may either be defined as the smallest or the largest word under $\rtri$, but in this study, the placement of $\algI$ in the said ordering is immaterial. 

Given a word $w$ on $\genset$, denote by $\endings_w$ the set of all proper endings of $w$. The word $w$ is \emph{regular} if, for any $b\in\endings_w$, if $a$ is the proper beginning of $w$ such that $w=ab$, then $w\rtri ba$. Note that the elements of $\genset$ are vacuously regular. Let $\regBasis$ be the set of all regular words on $\genset$. Arguably the most important property of regular words is in the following.

\begin{proposition}[{\cite[Theorem 2.8.3]{ufn95}}] Let $w\in\regBasis$, and let $d\in\endings_w$ such that 
\begin{eqnarray}
|d| = \max\{|b|\  :\  b\in\regBasis\cap\endings_w\}.\label{maxend}
\end{eqnarray}
If $c$ is the proper beginning of $w$ such that $w=cd$, then $c\rtri d$ and $c\in\regBasis$.
\end{proposition}
That is, if $d$ is the length-maximal proper ending of $w$ that is also regular, then the beginning $c$ that remains after removing $d$ from $w$ is also regular. Also, the maximality condition \eqref{maxend} for $d$ implies that the ``factoring'' $w=cd$ of $w$ is unique. We shall call this \emph{the regular factoring} of $w$. This notion allows us to define the following important concept relating regular words to Lie polynomials.

\begin{definition} For each $x\in\genset$, we define $\lreg x\rreg := x\in\freeLie$. Let $w\in\regBasis$ with regular factoring $w=cd$. Suppose that for any $w'\in\regBasis$ with $|w'|< |w|$, the Lie polynomial $\lreg w'\rreg$ has been defined. Then 
\begin{eqnarray}
\lreg w\rreg:=\lbrak\lreg c\rreg,\lreg d\rreg\rbrak.\nonumber
\end{eqnarray}
We call $\lreg w\rreg$ the \emph{bracketing} of the regular word $w$. Also, for any $\arbregsub\sub\regBasis$, we define
\begin{eqnarray}
\lreg \arbregsub\rreg := \{\lreg w\rreg\  :\  w\in \arbregsub\}.\nonumber
\end{eqnarray}
\end{definition}
The significance of regular words is because of the following classical result.
\begin{proposition}[{\cite[Theorem 2.1]{shi09a}}] The elements of $\lreg\regBasis\rreg$ form a basis for the $\F$-vector space $\freeLie$.
\end{proposition}

\paragraph{Regular words on two generators} From this point onward, we assume that $|\genset|=2$. Denote the distinct elements by $u$ and $v$, and we use the ordering $v>u$ in constructing the regular words. We mention the following useful property of regular words.

\begin{proposition}[{\cite[Theorem 2.8.1]{ufn95}}]\label{prodProp} If the words $a$ and $b$ are regular and $a\rtri b$, then $ab$ is regular.
\end{proposition}

The generators $v$ and $u$ are regular with $v\rtri u$. By Proposition~\ref{prodProp}, this is sufficient to conclude that the product $vu$ is regular. Proceeding by induction, any word of the form $v^mu^n$, with $m,n\in\Zplus$, is hence regular.

In particular, given $n_1,n_2\in\Zplus$, each of the words $vu^{n_1}$ and $vu^{n_2}$ is regular. If we further have $n_1<n_2$, then by comparison of the letters at the $(n_1+2)$-th position in each of the words $vu^{n_1}\cdot vu^{n_2}$ and $vu^{n_2}\cdot vu^{n_1}$, we have
\begin{eqnarray}
vu^{n_1}\cdot vu^{n_2} >  vu^{n_2}\cdot vu^{n_1}. \nonumber
\end{eqnarray}
Then $vu^{n_1}\rtri vu^{n_2} $, and by Proposition~\ref{prodProp}, the word $vu^{n_1}vu^{n_2}$ is regular. As will be relevant in later proofs, we collect these special types of regular words into the sets 
\begin{eqnarray}
%\vmun & := & \{v^mu^n\  :\  m,n\in\Zplus\},\nonumber\\
\vumvun & := & \{vu^{n_1}vu^{n_2}\  :\  n_1,n_2\in\N,\  \  n_1<n_2,\  \  n_2\geq 1\},\nonumber\\
\vun & := & \{vu^n\  :\  n\in\Zplus\},\nonumber
\end{eqnarray} 
which are actually subsets of $\regBasis$. %Note that if $n_2\in\Zplus$, then $v^2u^{n_2}\in\vmun\cap\vumvun$.

If $\arbLie$ is a Lie algebra over $\F$ and if $x\in\arbLie$, the map $\ad x:\arbLie\rightarrow\arbLie$ is the linear map defined by the rule $y\mapsto\lbrak x,y\rbrak$. If $m,n\in\Zplus$, then using the definition of regular words and the Lie algebra axioms, the identities
\begin{eqnarray}
\lreg v^mu^n\rreg & = & \lpar\lpar\ad v\rpar^{m-1}\circ\lpar -\ad u\rpar^n\rpar\lpar v\rpar,\label{vmunbrak}\\
\lreg v^2u^{n+1}\rreg & = & \lbrak\lreg v^2u^n\rreg ,u\rbrak - \lreg vuvu^n\rreg,\quad\quad\quad\quad\quad\quad\   (n\geq 2),\label{vvunind}\\
%\lreg v^2u^2\rreg & = & \lbrak\lreg v^2u\rreg,u\rbrak,\\
\lreg vu^mvu^n\rreg & = & \lbrak\lreg vu^m\rreg,\lreg vu^n\rreg\rbrak,\label{vumvunbrak}\\
\lreg vu^mvu^{n+1}\rreg & = & \lbrak \lreg vu^mvu^{n}\rreg,u\rbrak - \lreg vu^{m+1}vu^{n}\rreg,\quad (m\leq n-2),\label{vumvunind}\\
\lreg vu^{n-1}vu^{n+1}\rreg & = & \lbrak\lreg vu^{n-1}vu^{n}\rreg,u\rbrak,\quad\quad\quad\quad\quad\quad\quad\   \  (n\geq 2)\label{nminus1}
\end{eqnarray}
hold in $\freeLie$. The proofs involve elementary arguments proceeding from the definition of a regular word, and also some routine computations and induction.

\begin{lemma}\label{skewvumvunLem} For any $w_1,w_2\in\genset\cup\vun$, there exists $\varepsilon\in\{-1,0,1\}$ and some $w\in\vumvun$ such that
\begin{eqnarray}
\lbrak\lreg w_1\rreg,\lreg w_2\rreg\rbrak =\varepsilon \lreg w\rreg.\label{vumvuneq}
\end{eqnarray}
\end{lemma}
\begin{proof} We first deal with the case $w_1,w_2\in\vun$, in which, there exist $n_1,n_2\in\Zplus$ such that $w_1=vu^{n_1}$ and $w_2=vu^{n_2}$. If $n_1<n_2$, then $w=w_1w_2\in\vumvun\sub\regBasis$, and according to \eqref{vumvunbrak}, we have $\lreg w\rreg=\lbrak \lreg w_1\rreg,\lreg w_2\rreg\rbrak$. By taking $\varepsilon=1$, we get \eqref{vumvuneq}. If $n_1>n_2$, we take $\varepsilon = -1$, and $w=w_2w_1\in\vumvun\sub\regBasis$. We get \eqref{vumvuneq} by \eqref{vumvunbrak} and the skew-symmetry of the Lie bracket. If $n_1=n_2$, then take $\varepsilon=0$ and any $w\in\vumvun$. This completes the proof for our first case. The remaining cases to be checked are when $w_1\in\genset$ \emph{or} $w_2\in\genset$. If both conditions hold, then we take $w=vu$ and $\varepsilon$ can be any element of $\{-1,0,1\}$ depending on the combination of values of $w_1$ and $w_2$. It is routine to check these specific cases, and make use of the Lie algebra axioms. Because of the skew-symmetry of the Lie bracket, we are only left with the case $w_1=u$ and $w_2\in\vun$ and the case $w_1=v$ and $w_2\in\vun$. In either case, there exists $n\in\Zplus$ such that $w_2=vu^n$. If $w_1=u$, then we take $\varepsilon=-1$, and $w=vu^{n+1}$, with regular factoring $w=vu^n\cdot u$, and such that
\begin{eqnarray}
\varepsilon\lreg w\rreg =-\lreg vu^{n+1}\rreg =-\lbrak\lreg vu^n\rreg ,u\rbrak=\lbrak u,\lreg vu^n\rreg \rbrak=\lbrak\lreg w_1\rreg,\lreg w_2\rreg\rbrak .\nonumber
\end{eqnarray}
By a similar argument and similar computations, for the case $w_1=v$ and $w_2=vu^n$, we take $w=v^2u^n$ and $\varepsilon=1$.
\end{proof}

\section{Reordering in the algebra $\envalg$, and some consequences}

The reordering formula \eqref{xyReorder}, by induction, can be generalized into
\begin{eqnarray}
y^lx^m = x^m (y-m)^l.\label{induct1}
\end{eqnarray}
Given $a\in\envalg$, we denote by $\F\lbrak a\rbrak$, as is customary, the subalgebra of $\envalg$ generated by $a$, which is isomorphic to the polynomial algebra over $\F$ in one variable. The expression $(y-m)^l$ in the right-hand side of \eqref{induct1} is an element of $\F\lbrak y\rbrak$, and so the usual binomial theorem applies to this element, and we further have
\begin{eqnarray}
y^lx^m=\sum_{t=0}^l{{l}\choose{t}}(-m)^{l-t}x^my^t.\label{induct2}
\end{eqnarray}
The significance of the reordering formula \eqref{induct2} is that it gives us the structure constants of the algebra $\envalg$ with respect to the basis \eqref{Ubasis}. That is, if we take two arbitrary basis elements $x^ky^l$ and $x^my^n$ from \eqref{Ubasis}, we have
\begin{eqnarray} x^ky^l\cdot x^my^n = x^k(y^lx^m)y^n = \sum_{t=0}^l{{l}\choose{t}}(-m)^{l-t}x^{k+m}y^{t+n}.\label{structalg}
\end{eqnarray}
By some routine computations that make use of \eqref{structalg}, the structure constants of $\envalg$ as a Lie algebra are given by the relation
\begin{eqnarray}
\lbrak x^ky^l,x^my^n\rbrak =\sum_{t=0}^l{{l}\choose{t}}(-m)^{l-t}x^{k+m}y^{t+n}-\sum_{s=0}^n{{n}\choose{s}}(-k)^{n-s}x^{k+m}y^{s+l}.\label{structLie}
\end{eqnarray}
Setting $l=1=n$ in \eqref{structLie}, we have
\begin{eqnarray}
\lbrak x^ky,x^my\rbrak = (k-m)x^{k+m}y.\label{structLie1}
\end{eqnarray}
Let $\arbsubF$ be a subfield of $\F$ and let $S\sub\envalg$. Throughout, by the $\arbsubF$-linear span of $S$, or in symbols $\spanArb S$, we refer to the set of all \emph{finite} linear combinations of elements of $S$ (regardless of the cardinality of $S$) with scalar coefficients from $\arbsubF$. Under this interpretation, the algebra $\envalg$ being the $\F$-linear span of the basis \eqref{Ubasis} is consistent with the formulation of the Diamond Lemma \cite{ber78}, which was used to obtain the said basis from the standard presentation of $\envalg$. The relation \eqref{structLie1} implies that the $\F$-linear span $\uvLie$ of all elements of $\envalg$ of the form
\begin{eqnarray}
x^ky,\quad\quad\quad\quad (k\in\Zplus),\label{Rbasis}
\end{eqnarray}
is closed under the Lie bracket, and is hence a Lie subalgebra of $\envalg$. 

For the next computations, we shall make use of the following.
\begin{assumption} The field $\F$ contains a subfield isomorphic to $\Q$.
\end{assumption}
Thus, $\fchar\F=0$ and the nonzero integers have multiplicative inverses. Using \eqref{structLie1} and induction, the relation
\begin{eqnarray}
x^{n+2}y=\frac{1}{n!}(-\ad xy)^n(x^2y)\label{Rbasisrel}
\end{eqnarray}
holds in $\envalg$ for any $n\in\N$. Based on \eqref{Rbasisrel}, every basis element  of $\uvLie$ from \eqref{Rbasis} is a Lie polynomial in $xy$ and $x^2y$. Since $\uvLie$ is generated by two elements, there exists a Lie ideal $\defId$ of the two-generator free Lie algebra $\freeLie$ such that $\uvLie=\freeLie/\defId$, and that the kernel of the canonical Lie algebra homomorphism $\freeLie\rightarrow\uvLie$ with $u\mapsto xy$ and $v\mapsto x^2y$ is precisely $\defId$. In view of the isomorphism $\uvLie=\freeLie/\defId$, we identify $u$ with $xy$ and $v$ with $x^2y$. Thus, from \eqref{Rbasisrel}, and also from \eqref{vmunbrak}, we have
\begin{eqnarray}
x^{n+2}y= \frac{1}{n!}(-\ad u)^n(v)=\frac{1}{n!}\lreg vu^n\rreg,\quad\quad(n\in\N).\label{xytouv}
\end{eqnarray}
A consequence of \eqref{xytouv} is that the elements
\begin{eqnarray}
u,\quad\lreg vu^n\rreg,\quad\quad\quad\quad (n\in\N),\label{uvLieBasis}
\end{eqnarray}
form a basis for $\uvLie$. Using \eqref{vumvunbrak}, \eqref{structLie1}, and \eqref{Rbasisrel}, and by routine computations, we have
\begin{eqnarray}
\lreg vu^mvu^n\rreg = (m-n)\frac{m!n!}{(m+n+2)!}\lreg vu^{m+n+2}\rreg,\quad\quad (m\in\N,\  n\in\Zplus).\label{vumvunrel}
\end{eqnarray}
As will be relevant in later arguments, setting $m=n-1$ in \eqref{vumvunrel}, we find that the relations
\begin{eqnarray}
 \lreg vu^{n-1}vu^n\rreg + \frac{n!(n-1)!}{(2n+1)!}\lreg vu^{2n+1}\rreg=0,\quad\quad\quad\quad(n\in\Zplus),\label{defrels}
\end{eqnarray}
hold in $\uvLie$. 

\begin{proposition} The direct sum $\xpoly\dirsum\uvLie$ is a Lie subalgebra of $\envalg$.
\end{proposition}
\begin{proof} Showing that $\xpoly\dirsum\uvLie$ is closed under the Lie bracket will suffice. Because we have already shown that $\uvLie$ is a Lie subalgebra of $\envalg$, using the Lie algebra axioms, and the fact that $\xpoly$ is a commutative subalgebra of $\envalg$, we only need to show that the Lie bracket of an element of $\xpoly$ with an element of $\uvLie$ is an element of $\xpoly\dirsum\uvLie$. Because of the bilinearity of the Lie bracket, we only consider basis elements. Among the basis elements \eqref{Ubasis} of $\envalg$, those with $l=0$ form a basis for $\xpoly$, while those with $l=1$ form a basis for $\uvLie$.  Hence, the sum $\xpoly\dirsum\uvLie$ is indeed direct. Among these basis elements, let $a$ be a basis element of $\xpoly$, and $b$ be a basis element of $\uvLie$. Then there exist $k,m\in\Zplus$ such that $a=x^k$ and $b=x^my$. By the skew-symmetry of the Lie bracket, we only need to show $\lbrak a,b\rbrak\in\xpoly\dirsum\uvLie$. By \eqref{structLie}, we have
\begin{eqnarray}
\lbrak a,b\rbrak = \lbrak x^k,x^my\rbrak=kx^{k+m}\in\xpoly\dirsum\uvLie.\qedhere\nonumber
\end{eqnarray}
\end{proof}

For each $n\in\N$, define $\filter_n$ as the $\F$-linear span of all basis elements $x^ky^l$ of $\envalg$ from \eqref{Ubasis} such that $l\leq n$. As an immediate observation, note that $\filter_0=\xpoly$ and $\filter_1=\xpoly\dirsum\uvLie\dirsum\span\{y\}$.

If $\arbLie_1$ and $\arbLie_2$ are vector subspaces of a Lie algebra $\arbLie$ over $\F$, then by $\lbrak \arbLie_1,\arbLie_2\rbrak$ we mean the $\F$-linear span of all elements $\lbrak a,b\rbrak$ such that $a\in\arbLie_1$ and $b\in\arbLie_2$. A sequence $\lpar \arbLie_n\rpar_{n=0}^\infty$ of vector subspaces of $\arbLie$ is said to an \emph{$\N$-filtration} of $\arbLie$ if,  for any $p,q\in\N$, we have
\begin{eqnarray} 
\arbLie_p & \sub & \arbLie_{p+1},\nonumber\\
\lbrak\arbLie_p,\arbLie_q\rbrak & \sub & \arbLie_{p+q},\nonumber\\
\bigcup_{n=0}^\infty\arbLie_n & = & \arbLie.\nonumber
\end{eqnarray}
\begin{proposition} The sequence $\lpar\filter_n\rpar_{n=0}^\infty$ of vector subspaces of $\uvLie$ is an $\N$-filtration of $\uvLie$.
\end{proposition}
\begin{proof} Let $p,q\in\N$. The relations $\filter_p\sub\filter_{p+1}$ and $\displaystyle\bigcup_{n=0}^\infty\filter_n=\uvLie$ follow from the definition of the vector subspaces in  $\lpar\filter_n\rpar_{n=0}^\infty$, while the relation $\lbrak\filter_p,\filter_q\rbrak\sub\filter_{p+q}$ follows immediately from \eqref{structLie}.
\end{proof}

\section{A presentation for the Lie algebra $\uvLie$}

Our goal is to obtain a presentation for $\uvLie$ by generators and relations. We start with some computations and results in the free Lie algebra $\freeLie$. For each $n\in\Zplus$, we define
\begin{eqnarray}
\relcoeff_n & := & \frac{n!(n-1)!}{(2n+1)!}\in\Q,\nonumber\\
\infgen_n & := & \lreg vu^{n-1}vu^n\rreg + \relcoeff_n\lreg vu^{2n+1}\rreg\in\freeLie.\nonumber
\end{eqnarray}
Denote by $\infId$ the Lie ideal of $\freeLie$ generated by $\{\infgen_n\  :\  n\in\Zplus\}$. 

\begin{remark}\label{subRem} Since the relations \eqref{defrels} hold in $\uvLie$, for each $n\in\Zplus$, we have $\infgen_n\in\defId$. Consequently, $\infId\sub\defId$.
\end{remark}
 Towards the other set inclusion, we have a couple of lemmas.

\begin{lemma}\label{vumvunLem} For each $w\in\vumvun$, there exists $\beta\in\infId$ such that 
\begin{eqnarray}
\lreg w\rreg-\beta\in\spanQ\lpar\genset\cup\lreg\vun\rreg\rpar.
\end{eqnarray}
\end{lemma}
\begin{proof} Let $w\in\vumvun$. Then there exist $m,n\in\N$ with $n\geq 1$ such that $w=vu^mvu^n$. We first deal with the case $m\neq 0$, and here we use induction on $n$. The smallest possible value of $n$ in such a case is $n=2$, and for $vu^mvu^n$ to be regular, the only possibility is that $m=1$. Using the definition of $\alpha_2$, we have
\begin{eqnarray} 
\lreg vuvu^2\rreg = \infgen_2 - a_2\lreg vu^5\rreg,
\end{eqnarray}
where $\infgen_2\in\infId$ and $a_2\lreg vu^5\rreg\  \in\  \spanQ\lpar\genset\cup\lreg\vun\rreg\rpar$. Suppose that for some $n\in\Zplus$, all regular words of the form $vu^mvu^n$ satisfy the statement. We proceed with taking sub-cases depending on the value of $m$. First, we consider the case $m\leq n-2$. In this case, both $m$ and $m+1$ are strictly less than $n$. Then the words $vu^mvu^n$ and $vu^{m+1}vu^n$ are regular and both satisfy the inductive hypothesis. Thus, there exist $\beta_1,\beta_2\in\infId$, some $c_1,c_2,\ldots,c_p,e_1,e_2,\ldots,e_q\in\Q$, and some $h_1,h_2,\ldots,h_p,k_1,k_2,\ldots,k_q\in\Zplus$ such that 
\begin{eqnarray}
\lreg vu^mvu^n\rreg & = & \beta_1 + \sum_{s=1}^pc_s\lreg vu^{h_s}\rreg,\label{nind1}\\
\lreg vu^{m+1}vu^n\rreg & = & \beta_2 + \sum_{t=1}^qe_t\lreg vu^{k_t}\rreg. \label{nind2}
\end{eqnarray}
We substitute \eqref{nind1} and \eqref{nind2} into \eqref{vumvunind}. In the resulting equation, Lie polynomials of the form $\lbrak\lreg vu^{h_s}\rreg,u\rbrak$ appear. Regarding these Lie polynomials, observe that the regular factoring of the regular word $vu^{h_s+1}$ is $vu^{h_s+1}=vu^{h_s}\cdot u$, and so $\lbrak\lreg vu^{h_s}\rreg,u\rbrak=\lreg vu^{h_s+1}\rreg$. From these observations we get the simplified equation
\begin{eqnarray}
\lreg vu^mvu^{n+1}\rreg -\lpar \lbrak\beta_1, u\rbrak -\beta_2 \rpar & = & \sum_{s=1}^pc_s\lreg vu^{h_s+1}\rreg - \sum_{t=1}^qe_t\lreg vu^{k_t}\rreg,\label{similareq}
\end{eqnarray}
where $ \lbrak\beta_1, u\rbrak -\beta_2 \in\infId$, and the right-hand side of \eqref{similareq} is in $\spanQ\lpar\genset\cup\lreg\vun\rreg\rpar$. If $m=n-1$, then we substitute \eqref{nind1} into \eqref{nminus1}, and we obtain
\begin{eqnarray}
\lreg vu^mvu^{n+1}\rreg - \lbrak\beta_1, u\rbrak  & = & \sum_{s=1}^pc_s\lreg vu^{h_s+1}\rreg\quad\in\quad\spanQ\lpar\genset\cup\lreg\vun\rreg\rpar,\nonumber
\end{eqnarray}
where $\lbrak\beta_1, u\rbrak\in\infId$. If $m=n$, then we simply use the definition of $\infgen_{n+1}$ to obtain
\begin{eqnarray}
\lreg vu^mvu^{n+1}\rreg - \infgen_{n+1} & = & -\relcoeff_{n+1}\lreg vu^{2n+3}\rreg\quad\in\quad\spanQ\lpar\genset\cup\lreg\vun\rreg\rpar.\nonumber
\end{eqnarray}
This completes the induction for the case $m\neq 0$. Suppose $m=0$. The smallest possible value of $n$ is $1$, and from the definition of the generators of $\infId$, we have $\lreg v^2u\rreg - \infgen_1 = a_1\lreg vu^3\rreg\  \in\  \spanQ\lpar\genset\cup\lreg\vun\rreg\rpar$. Suppose the statement holds for some $n\in\Zplus$ with $m=0$. Then there exists $\gamma_1\in\infId$, some $f_1,f_2,\ldots,f_P\in\Q$, and some $m_1,m_2,\ldots,m_P\in\Zplus$ such that 
\begin{eqnarray}
\lreg v^{2}u^n\rreg & = & \gamma_1 + \sum_{s=1}^Pf_s\lreg vu^{m_s}\rreg. \label{nind3}
\end{eqnarray}
Since we have proven the statement for the case $m\neq 0$, which is applicable to the regular word $vuvu^n$, there exists $\gamma_2\in\infId$,  some $g_1,g_2,\ldots,g_Q\in\Q$, and some $n_1,n_2,\ldots,n_Q\in\Zplus$ such that 
\begin{eqnarray}
\lreg vuvu^n\rreg & = & \gamma_2 + \sum_{t=1}^Qg_s\lreg vu^{n_s}\rreg. \label{nind4}
\end{eqnarray}
Similar to the proof for the case $m\neq 0$ with $m\leq n-2$, we substitute \eqref{nind3} and \eqref{nind4} into \eqref{vvunind}. We get an equation similar to \eqref{similareq}, which asserts that 
\begin{eqnarray}
\lreg v^2u^{n+1}\rreg-\lpar\lbrak\gamma_1,u\rbrak-\gamma_2\rpar\quad\in\quad\spanQ\lpar\genset\cup\lreg\vun\rreg\rpar,\nonumber
\end{eqnarray}
where $\lbrak\gamma_1,u\rbrak-\gamma_2\in\infId$.
\end{proof}
Our next step is to extend Lemma~\ref{vumvunLem} from $\vumvun$ to $\regBasis$.

\begin{lemma}\label{vunLem} For each $w\in\regBasis$, there exists $\beta\in\infId$ such that 
\begin{eqnarray}
\lreg w\rreg-\beta\in\spanQ\lpar\genset\cup\lreg\vun\rreg\rpar.
\end{eqnarray}
\end{lemma}
\begin{proof} Let $w\in\regBasis$. We use induction on $|w|$. If $|w|=1$, then take $\beta=0\in\infId$ and use the fact that $\lreg w\rreg=w\in\genset$. Suppose that for some $n\in\Zplus$, all regular words of length strictly less than $n$ satisfy the statement. Let $w$ be any regular word with $|w|=n$. If $w=ab$ is the regular factoring of $w$, then using the inductive hypothesis, we write the regular words $a$ and $b$ as
\begin{eqnarray}
\lreg a\rreg - \beta_1 & = & \sum_{s=1}^p c_s\lreg x_s\rreg,\label{beginreg}\\
\lreg b\rreg - \beta_2 & = & \sum_{t=1}^q e_t\lreg y_t\rreg,\label{endreg}
\end{eqnarray}
for some $\beta_1,\beta_2\in\infId$, where $c_s,e_t\in\Q$ and $x_s,y_t\in\genset\cup\vun$ for all $s,t$. Substituting \eqref{beginreg} and \eqref{endreg} into the equation $\lreg w\rreg=\lbrak\lreg a\rreg,\lreg b\rreg\rbrak$, and using the bilinearity of the Lie bracket, we obtain
\begin{eqnarray}
\lreg w\rreg-\beta=\sum_{s=1}^p\sum_{t=1}^q\lpar c_se_t\rpar\lbrak\lreg x_s\rreg,\lreg y_t\rreg\rbrak,\label{weq}
\end{eqnarray}
where $\beta = \lbrak\beta_1,\beta_2\rbrak + \lbrak\beta_1,\sum_{t=1}^q e_t\lreg y_t\rreg\rbrak + \lbrak \sum_{s=1}^p c_s\lreg x_s\rreg,\beta_2\rbrak\  \in\  \infId$. Consider an arbitrary term $\lpar c_se_t\rpar\lbrak\lreg x_s\rreg,\lreg y_t\rreg\rbrak$ in the right-hand side of \eqref{weq}. The scalar coefficient is in $\Q$. By a routine check on all the possible cases, the Lie polynomial $\lbrak\lreg x_s\rreg,\lreg y_t\rreg\rbrak$ is one of
\begin{eqnarray}
\lbrak u,v\rbrak,\quad\lbrak v,u\rbrak,\label{invun1}\\
\lbrak u,\lreg vu^{n_2}\rreg\rbrak,\quad\lbrak\lreg vu^{n_1}\rreg,u\rbrak,\label{invun2}\\
\lbrak v,\lreg vu^{n_2}\rreg\rbrak,\quad\lbrak\lreg vu^{n_1}\rreg,v\rbrak,\label{invumvun1}\\
\lbrak\lreg vu^{n_1}\rreg,\lreg vu^{n_2}\rreg\rbrak,\label{invumvun2}
\end{eqnarray}
for some $n_1,n_2\in\Zplus$. By some routine computations using the skew-symmetry of the Lie bracket and the properties of the bracketing of elements of $\vun$ and of $\vumvun$, each Lie polynomial of the form \eqref{invun1} or \eqref{invun2} is equal to $\varepsilon z_1$ for some $\varepsilon\in\{-1,1\}$ and some $z_1\in\vun$, and that each Lie polynomial of the form \eqref{invumvun1} is equal to $\delta z_2$ for some $\delta\in\{-1,1\}$ and some $z_2\in\vumvun$. By Lemma~\ref{skewvumvunLem}, each Lie polynomial of the form \eqref{invumvun2} is equal to $\eta z_3$ for some $\eta\in\{-1,0,1\}$ and some $z_3\in\vumvun$. By these observations about \eqref{invun1} to \eqref{invumvun2}, we can rewrite \eqref{weq} as $\lreg w\rreg-\beta=w_1+w_2$ for some $w_1\in\spanQ\lreg \vun\rreg$ and some $w_2\in\spanQ\lreg\vumvun\rreg$. But  by Lemma~\ref{vumvunLem}, there exists $\gamma\in\infId$ such that $w_2=\gamma + w_3$ for some $w_3\in\spanQ\lreg \vun\rreg$. Then $\lreg w\rreg - (\beta+\gamma)=w_1+w_3\in\spanQ\lreg \vun\rreg$, where $\beta+\gamma\in\infId$. This completes the induction.
\end{proof}

\begin{theorem} The Lie algebra $\uvLie$ has a presentation with generators $u,v$ and relations $\infgen_n=0$ for all $n\in\Zplus$.
\end{theorem}
\begin{proof} Showing that the elements in $\{\infgen_n\  :\  n\in\Zplus\}$ generate the defining Lie ideal $\defId$ of $\uvLie$, or equivalently $\infId=\defId$, will suffice. In view of Remark~\ref{subRem}, we only need to show $\defId\sub\infId$. Let $w\in\defId$. We write $w$ as an $\F$-linear combination of the basis elements of $\freeLie$ from $\lreg\regBasis\rreg$, and to each basis element that is not in $\lreg\vun\rreg$ appearing in the linear combination we apply Lemma~\ref{vunLem}. The result is
\begin{eqnarray}
w-\beta = \sum_{s=1}^pc_s\lreg vu^{n_s}\rreg\label{zerocoeff}
\end{eqnarray}
for some $\beta\in\infId$, where $c_s\in\F$ for all $s$. Applying the canonical Lie algebra homomorphism $\freeLie\rightarrow\uvLie$ on both sides of \eqref{zerocoeff}, we obtain $0=\sum_{s=1}^pc_s\lreg vu^{n_s}\rreg$ where $\lreg vu^{n_s}\rreg$ is a basis element of $\uvLie$ from \eqref{uvLieBasis} for any $s$. Then $c_s=0$ for any $s$, and substituting these values to \eqref{zerocoeff}, we get $w=\beta\in\infId$.
\end{proof}

%\section*{References}

\bibliography{mybibfile}

\end{document}